\documentclass[12pt,reqno,intlim]{amsart}
\usepackage{amssymb,amsmath}

\usepackage[arrow,matrix,curve]{xy}

\newdir{ >}{{}*!/-5pt/\dir{>}}

\newcommand{\nc}{\newcommand}

\nc{\bC}{\bold{C}} \nc{\bN}{\Bbb{N}} \nc{\cF}{\mathcal{F}}
\nc{\cE}{\mathcal{E}} \nc{\cR}{\mathcal{R}} \nc{\cM}{\mathcal{M}}
\nc{\al}{\alpha} \nc{\bt}{\beta} \nc{\gm}{\gamma} \nc{\dl}{\delta}
\nc{\om}{\omega} \nc{\sg}{\sigma} \nc{\Sg}{\Sigma} \nc{\vf}{\varphi}
\nc{\ve}{\varepsilon} \nc{\os}{\overset} \nc{\ol}{\overline}
\nc{\ul}{\underline} \nc{\us}{\underset} \nc{\sbs}{\subset}
\nc{\bsl}{\backslash} \nc{\Ra}{\Rightarrow}
\nc{\lra}{\longrightarrow} \nc{\all}{\allowdisplaybreaks}

\nc{\Codes}{\operatorname{{\bold{Codes}}}}
\nc{\RegMono}{\operatorname{\mathcal{R}{\rm{eg}\mathcal{M}{\rm{ono}\!}}}}
\nc{\RegEpi}{\operatorname{\mathcal{R}{\rm{eg}\mathcal{E}{\rm{pi}\!}}}}
\nc{\Mn}{\operatorname{\mathcal{M}{\rm{ono}\!}}}
\nc{\Ep}{\operatorname{\mathcal{E}{\rm{pi}\!}}}
\nc{\Rg}{\operatorname{\mathcal{R}{\rm{eg}\!}}}
\nc{\Ob}{\operatorname{Ob\!}}

\numberwithin{equation}{section}

\newtheorem{theo}{\ \ \ Theorem}[section]
\newtheorem{lem}[theo]{\ \ \ Lemma}
\newtheorem{prop}[theo]{\ \ \ Proposition}
\newtheorem{cor}[theo]{\ \ \ Corollary}

\theoremstyle{definition}

\theoremstyle{remark}

\begin{document}

\title[]
{Ideals in BIT speciale varieties}

\author{Dali Zangurashvili}

\maketitle

% Abstract.
\begin{abstract}
Ideals in BIT speciale varieties are characterized. In particular, it is proved that, for any finitary BIT speciale variety, there is a finite set of ideal terms determining ideals. Several ideal term sets of this kind are given. For the variety of groups one of these sets consists of the terms $y_1y_2$, $xyx^{-1}$, $x^{-1}y^{-1}x$ (and $y$ which can be ignored), while for rings it consists of the terms $y_1+y_2$, $-y$, $xy$, $yx$. For each of the following varieties -- groups with multiple operators, semi-loops, and divisible involutory groupoids -- one of the term sets found in this paper almost coincides with the term sets found earlier for these particular varieties by resp. Higgins, B\v{e}lohl\'{a}vek and Chajda, and Hala\v{s}. The coincidence is precise in the case of divisible involutory groupoids. For loops (resp. loops with operators), the intersection of one of the term sets found in this paper with the term sets found earlier for loops (resp. for loops with operators) by Bruck (resp. by Higgins) forms the major part of both sets.


\end{abstract}

% 1.
\section{Introduction}

As is well-known, for some particular pointed varieties of universal algebras, there are finite sets of terms $t_i(x_1,x_2,...,x_{m_{i}},y_1,y_2,...y_{n_{i}})$ such that normal subalgebras (i.e. kernels of homomorphisms) can be characterized as follows:\vskip+2mm

(*) a non-empty subset $H$ of an algebra $A$ is a normal subalgebra if and only if it is closed under all terms $t_i$ for arbitrary fixed values of the variables $x_1,x_2,...,x_{m_{i}}$, i.e. for any $i$, any $a_1,a_2,...,a_{m_{i}}\in A$, and any $h_1,h_2,...,h_{n_{i}}\in H$, one has $t_i(a_1,a_2,...,a_{m_{i}},h_1,h_2,...,h_{n_{i}})\in H$.\vskip+2mm

\noindent For instance, for groups such a set consists of the terms  $y_1y_2^{-1}$, $xyx^{-1}$, for rings it consists of the terms $y_1-y_2$, $xy$ and $yx$. In the present paper we show that a finite set of terms $t_i$ satisfying the condition (*) exists for any finitary semi-abelian variety. In fact, we deal with more general case of a finitary BIT speciale variety, and in place of normal subalgebras consider ideals (which are precisely normal subalgebras in the semi-abelian case).

The notion of a semi-abelian variety, and more generally, of a semi-abelian category was introduced by Janelidze, M\'arki and Tholen in 2002 \cite{JMT}. The investigations of a number of mathematicians  on the problem  "to find a list of axioms which reflect the properties of groups, rings and algebras as nicely as the abelian-category axioms do for abelian groups and modules \cite{JMT}" posed by MacLane \cite{M} in 1950 led to this notion. One of them was Bourn who introduced the notion of a protomodular category \cite{Bo}. If a category is Barr exact (like a variety of universal algebras) and has a zero object, then the protomodularity is equivalent to the split version of the Short Five Lemma \cite{Bo}. A semi-abelian category is defined as a Barr exact protomodular category with a zero object and finite coproducts. The notion of a semi-abelian category provides a good abstract foundation for the isomorphism and decomposition theorems, radical and commutator theory, homology of non-abelian structures. Any semi-abelian category (in particular, any semi-abelian variety) is a Mal'cev category.

Another class of Mal'cev varieties is that of BIT speciale (or classically ideal-determined) varieties. The notion of such a variety was introduced by Ursini at the beginning of 70's of the past century \cite{U1}, \cite{U2}, though the identities included in the definition of a BIT speciale variety (see (1.1) and (1.2) below) for the case where $n=1$ were considered earlier by Slomi\'nski \cite{S}. A BIT speciale variety is defined as a variety the algebraic theory of which contains a constant $0$ and, for some  natural $n$, binary terms $\alpha_{1}$,
$\alpha_{2}$,..., $\alpha_{n}$, and an $(n+1)$-ary term $\theta$ such that the following identities are satisfied:
\begin{equation}
\alpha_{i}(x,x)=0,
\end{equation}
\begin{equation}
\theta(\alpha_{1}(x,y),\alpha_{2}(x,y),...,\alpha_{n}(x,y),y)=x.
\end{equation} 

\noindent The equivalent notion is that of a $0$-coherent variety introduced by Beutler at the end of 70's \cite{B}.

The term \textit{BIT speciale variety} is of Italian origin. BIT is an abbreviation for "buona"+"ideali"+"teoria", or in sentence "buona teoria degli ideali", and is translated as a good theory of ideals, the precise meaning of which is explained in the Ursini's paper \cite{U2} (and Section 2 of this paper). 

The notion of an ideal also was introduced by Ursini \cite{U}, and took its roots from the earlier works by Higgins \cite{Hi}, Slomi\'nski \cite{S}, Magari \cite{Ma}, etc. (for details we refer the reader to \cite{GU} and \cite{JMT1}). Ideals were deeply studied by Gumm and Ursini in \cite{GU}.

Although the notion of a BIT speciale variety was introduced for purely universal-algebraic purposes, subsequently it turned out to be closely related to the notions of a protomodular category and of a semi-abelian category since, as it was proved by Bourn and Janelidze in \cite{BJ}, a variety is semi-abelian if and only if it is a pointed BIT speciale variety, (i.e. a BIT speciale variety such that $0$ is a unique constant in its algebraic theory). 

The varieties of groups, rings, loops, semi-loops, divisible involutory groupoids,  Heyting semi-lattices are semi-abelian, and hence BIT speciale. More examples of varieties of the latter kind are given by the varieties of groups with multiple operators, Boolean algebras, Heyting algebras.   

In this paper, for any BIT speciale variety, we give several term sets satisfying the condition (*), where "normal subalgebra" is replaced by "ideal". The terms in them are represented via the BIT special terms $0$, $\theta$, $\alpha_i$, and operations from the signature of the variety. Each of these sets is finite if the variety is finitary. For the varieties of groups and rings one of these sets contains the  typical terms. Namely, for groups, it consists of the terms $y_1y_2$ , $xyx^{-1}$, $x^{-1}y^{-1}x$ (and $y$, which, of course, can be ignored), while for the variety of rings it consists of the terms $y_1+y_2$, $-y$, $xy$ and $yx$. For each of the following varieties -- groups with multiple operators, semi-loops, and divisible involutory groupoids -- one of the term sets found in this paper almost coincides with the term sets given by resp. Higgins \cite{Hi}, B\v{e}lohl\'{a}vek and Chajda \cite{BC}, and Hala\v{s} \cite{H} for these particular varieties. The coincidence is precise in the case of the variety of divisible involutory groupoids. For loops (resp. loops with operators), the intersection of one of term sets found in this paper with the term sets given by Bruck for loops \cite{Br} (by Higgins \cite{Hi} for loops with operators) forms the major part of both sets.

Finally note that, for the semi-abelian case, there is a different characterization of ideals (=normal subalgebras). It is given by Mantovani and Metere \cite{MM}, and  by Hartl and Loiseau \cite{HL}. It is formulated in terms of commutators in abstract categories and works for arbitrary semi-abelian categories. In this relation we have to mention also the paper \cite{JMT1} by Janelidze, M\'arki, and Ursini, where the relationship between the basic constructions of semi-abelian category theory and the theory of ideals (and of clots) in universal algebra is studied.

\section{Ideal terms and ideals}

We begin with the definitions from the papers \cite{U} and \cite{GU}.

Let $\mathbb{V}$ be a variety of universal algebras with a signature $\cF$, and let $0$ be a constant in the algebraic theory of $\cF$. 

A term $t(x_1,x_2,...,x_m,y_1,y_2,...,y_n)$ over $\cF$ is called a 0-ideal in the variables $y_1,y_2,...,y_n$ if  
$$t(x_1,x_2,...,x_m,0,0,...,0)=0$$
 \noindent is an identity in $\mathbb{V}$. 
%We call variables $y_1,y_2,...,y_n$ ideal in this paper.

A non-empty subset $H$ of a $\mathbb{V}$-algebra $A$ is called a 0-ideal if, for any 0-ideal term $t(x_1,x_2,...,x_m,y_1,y_2,...,y_n)$ in the variables $y_1,y_2,...,y_n$, any $a_1,a_2,...,a_m\in A$ and any $b_1,b_2,...,b_n\in H$, one has $$t(a_1,a_2,...,a_m,b_1,b_2,...,b_n)\in H.$$

When no confusion might arise, we omit the prefix "0-".

One can easily verify that the kernel of any congruence on $A$ (i.e. the equivalence class containing $0$) is an ideal. 

A variety is called a BIT variety or ideal-determined if any ideal of its any algebra is the kernel of precisely one congruence. 

For the definitions of semi-abelian and  BIT speciale varieties we refer the reader to the Introduction. As it was mentioned there, a variety is semi-abelian if and only if it is BIT speciale and $0$ is the unique constant in its algebraic theory \cite{BJ}. The latter condition is equivalent to the one that the subset $\lbrace 0 \rbrace$ is a subalgebra in any algebra from this variety. In a semi-abelian variety ideals are precisely normal subalgebras (i.e. kernels of homomorphisms).

Any BIT speciale variety is a BIT variety, but the converse is not true \cite{U1}. The characterizations of the congruence corresponding to an ideal in a BIT speciale variety are given by the following observation (see Proposition 2.1 below). Before we give it, let us agree that, as usual, for a $\mathbb{V}$-algebra $A$, its subset $H$, a term $t(x_1,x_2,...,x_m,y_1,y_2,...,$ $y_n)$ and  $a_1,a_2,...,a_m\in A$, the symbol $t(a_1,a_2,...,a_m,H,H,...,H)$ denotes the set $\lbrace t(a_1,a_2,...,a_m,h_1,h_2,...,h_n)\vert h_1,h_2,...,h_n\in H\rbrace$.\vskip2mm
\begin{prop} (see e.g. \cite{B})
Let $\mathbb{V}$ be a BIT speciale variety, $A$ be its algebra, and $H$ be an ideal of $A$. Let $\simeq_H$ be the congruence whose kernel is $H$, and $a,b\in A$. Then the following conditions are equivalent:\vskip+2mm

(i) $a\simeq_H b$;\vskip+1mm

(ii) $\alpha_i(a,b)\in H$, for any $i$ ($1\leq i\leq n)$;\vskip+1mm

(iii) $a\in \theta(H,H,...,H,b)$;\vskip+1mm

(iv) $b\in \theta(H,H,...,H,a)$.

\end{prop}\vskip+2mm

Below we will characterize ideals in a BIT speciale variety. From now on we will assume that $\mathbb{V}$ is such a variety. \vskip+2mm

Proposition 2.1 immediately implies that for any $\mathbb{V}$-algebra $A$, its any ideal $H$, any $a\in A$, and any $i$ with $1\leq i\leq n$, we have the inclusion
\begin{equation}
\alpha_i(\theta(H,H,...,H,a),\theta(H,H,...,H,a))\subseteq H.
\end{equation}\vskip+2mm

Observe that, for a subset $H$ of an algebra $A$, the validity of inclusion (2.1) for any $a\in A$ and any $i$ with $1\leq i\leq n$ does not imply that $H$ is an ideal (for a counterexample take the variety of groups).

\vskip+4mm

Let $H$ be a non-empty subset of a $\mathbb{V}$-algebra $A$. In view of Proposition 2.1 we introduce the following equivalence relation on $A$:
\begin{center}
$a\sim_H b$ iff $\theta(H,H,...,H,a)=\theta(H,H,...,H,b)$.
\end{center}\vskip
+2mm

In the sequel we will repeatedly use the identity
\begin{equation}
\theta(0,0,...,0,a)=a.
\end{equation}
\noindent It immediately follows from (1.1) and (1.2).
\vskip+2mm

\begin{lem} Let $0\in H$. The equivalence class  $\left[  a\right] $ of any element $a$ of $A$ is contained in the set $\theta(H,H,...,H,a)$.
\end{lem}

\begin{proof}
Assume that $b$ is an element of $A$ with $a\sim_H b$. Since $b=\theta(0,0,...,0,b)\in \theta(H,H,...,H,b)$, we obtain that $b\in \theta(H,H,...,H,a)$.
\end{proof}

\begin{lem}
Let $H$ contains $0$ and is closed under the operations $\theta$ and $\alpha_i(-,0)$, for all $i$. Then 
\begin{equation}
\theta(H,H,...,H,0)=H.
\end{equation}
\end{lem}

\begin{proof} It suffices to observe that, for any $h\in H$ we have
$$h=\theta(\alpha_1(h,0),\alpha_2(h,0),...,\alpha_n(h,0),0).$$
\end{proof}

\begin{lem}
Let $H$ be a non-empty subset of $A$. The following conditions are equivalent and imply that $0\in H$:\vskip+1mm

(i) for any $a\in A$ and any $i$ with $1\leq i\leq n$, inclusion (2.1) holds;\vskip+1mm

(ii) for any $a\in A$ we have the equality
\begin{equation}
\left[  a\right] =\theta(H,H,...,H,a),
\end{equation}
and, moreover, for any $a,b\in A$, if $a\sim_H b$, then $\alpha_i(a,b)\in H$, for any $i$ $(1\leq i\leq n)$.
\end{lem}

\begin{proof} (i)$\Rightarrow$(ii): Since $H$ is not empty, inclusion (2.1) implies that $0\in H$.
Consider now an element
\begin{equation} 
b=\theta(h_1,h_2,...,h_n,a),
\end{equation}
\noindent with $h_1,h_2,...,h_n\in H$. We will show that 
\begin{equation}
\theta(H,H,...,H,a)=\theta(H,H,...,H,b).
\end{equation}
Consider any 
 $$c=\theta(h'_1,h'_2,...,h'_n,\theta(h_1,h_2,...,h_n,a))$$ 
 \noindent with $h'_1,h'_2,...,h'_n\in H$. Equality (1.2) implies that
$$\alpha_i(\theta(h'_1,h'_2,...,h'_n,\theta(h_1,h_2,...,h_n,a)),a)=$$
$$\alpha_i(\theta(h'_1,h'_2,...,h'_n,\theta(h_1,h_2,...,h_n,a)),$$
$$\theta(\alpha_1(a,\theta(h_1,h_2,...,h_n,a)),\alpha_2(a,\theta(h_1,h_2,...,h_n,a)),...$$
$$\alpha_n(a,\theta(h_1,h_2,...,h_n,a)), \theta(h_1,h_2,...,h_n,a))).$$
By (2.1), we have $\alpha_j(a,\theta(h_1,h_2,...,h_n,a))\in H$, for all $j$. Applying (2.1) again, we obtain that $\alpha_i(c,a)\in H$. Then from (1.2) we obtain that
$$c=\theta(\alpha_1(c,a),\alpha_2(c,a),...,\alpha_n(c,a),a)\in \theta(H,H,...,H,a).$$

\noindent This implies that 
$\theta(H,H,...,H,b)\subseteq \theta(H,H,...,H,a)$.

For the converse, consider again arbitrary $h'_1,h'_2,...,h'_n\in H$. We have
\begin{equation}
\theta(h'_1,h'_2,...,h'_n,a)=\theta(c_1,c_2,...,c_n,\theta(h_1,h_2,...,h_n,a)),
\end{equation}
\noindent where, for any $j$,  
$$c_j=\alpha_j(\theta(h'_1,h'_2,...,h'_n,a),\theta(h_1,h_2,...,h_n,a)).$$  
Inclusion (2.1) implies that $c_j\in H$. This together with (2.5) and (2.7) implies (2.6). Taking into account Lemma 2.2 we can conclude that equality (2.4) holds. 

If $a\sim_H b$, then $b=\theta(h_1,h_2,...,h_n,a)$, for some $h_1,h_2,...,h_n\in H$. Applying (2.1), we obtain that $\alpha_i(a,b)\in H$.

 The implication (ii)$\Rightarrow$(i) is obvious.

\end{proof}

\begin{theo}
Let $\mathbb{V}$ be a BIT speciale variety, and let $A$ be its any algebra. Then, for any non-empty subset $H$ of $A$, the following conditions are equivalent:

\vskip+2mm
(i) $H$ is an ideal;\vskip+1mm

(ii) $H$ is  closed under the operations $\theta$ and $\alpha_i$, for any $i$  with $1\leqslant i\leqslant n$, and, moreover, for any $\tau \in \cF\cup \lbrace \theta \rbrace$, any $i$  with $1\leqslant i\leqslant n$, and any $a_1,a_2,...,a_{k}\in A$ (with $k$ being the arity of $\tau$) we have:
\begin{equation}
\alpha_i(\tau(\theta(H,H,...,H,a_1),\theta(H,H,...,H,a_2),
\end{equation}
$$...,\theta(H,H,...,H,a_{k})),\tau(a_1,a_2,...,a_{k}))\subseteq H.$$\vskip+1mm

(iii)  $H$ contains $0$ and is  closed under the operations $\theta$ and $\alpha_i(-,0)$, for any $i$ with $1\leqslant i\leqslant n$, and, moreover, for any $\tau \in \cF\cup \lbrace \theta \rbrace \cup \lbrace \alpha_j\vert 1\leqslant j\leqslant n\rbrace$, any $i$ with $1\leqslant i\leqslant n$, and any $a_1,a_2,...,a_{k}\in A$ (with $k$ being the arity of $\tau$) we have (2.8);\vskip+1mm

(iv) $H$ is closed under the operations $\theta$ and $\alpha_i(-,0)$, for any $i$ with $1\leqslant i\leqslant n$, and, moreover, for any $\tau\in \cF$,  any $a_1,a_2,...,a_{k}\in A$ (with $k$ being the arity of $\tau$) and $i$ with $1\leqslant i\leqslant n$, we have (2.8) and
%\alpha_i(\theta(H,H,...,H,\theta(H,H,...,H,a)),a)\subseteq H,
%\end{equation}
\begin{equation}
\alpha_i(\theta(H,H,...,H,a),\theta(H,H,...,H,a))\subseteq H;
\end{equation}
\vskip+2mm

(v) $H$ is closed under the operations $\theta$ and $\alpha_i(-,0)$, for any $i$ with $1\leqslant i\leqslant n$, and, moreover, for any $\tau\in \cF$,  any $a_1,a_2,...,a_{k}\in A$ (with $k$ being the arity of $\tau$) and $i$, $j$ with $1\leqslant i\leqslant n, 1\leq j\leq k$, we have (2.9) and

\begin{equation}
\alpha_i(\tau(a_1, a_2,....,a_{j-1},\theta(H,H,...,H,a_j),a_{j+1},...
\end{equation}
$$...,a_{k}),\tau(a_1,a_2,...,a_{k}))\subseteq H;$$\vskip+1mm

(vi) $H$ contains $0$ and is closed under the operations $\theta$ and $\alpha_i(-,0)$, for any $i$ with $1\leqslant i\leqslant n$, and, moreover, for any $a\in A$ we have 
\begin{equation}
\left[  a\right] =\theta(H,H,...,H,a),
\end{equation} 
\noindent and moreover, for any $\tau\in \cF$,  any $a_1,a_2,...,a_{k}\in A$ (with $k$ being the arity of $\tau$) and $i$, $j$ with $1\leqslant i\leqslant n, 1\leq j\leq k$, we have inclusion (2.10);\vskip+1mm

(vii) $\sim_H$ is a congruence on $A$, and $H$ is its kernel.\vskip+2mm

If $\mathbb{V}$ is semi-abelian, then "$H$ is closed under the operations $\theta$ and $\alpha_i$, for any $i$ with $1\leqslant i\leqslant n$" in condition (ii), "$H$ contains $0$ and is closed under the operations $\theta$ and $\alpha_i(-,0)$, for any $i$ with  $1\leqslant i\leqslant n$" in condition (iii), "$H$ is closed under the operations $\theta$ and $\alpha_i(-,0)$, for any $i$ with $1\leqslant i\leqslant n$" in conditions (iv) and (v) can be replaced by "$H$ is a subalgebra of $A$".

\end{theo}

\begin{proof}

The implications (i)$\Rightarrow$(ii) and (i)$\Rightarrow$(iii) immediately follow from the definition of an ideal.

(ii)$\Rightarrow$(iv): Since $H$ is not empty and is closed under $\alpha_i$, it contains $0$. Moreover, inclusion (2.8) considered for $\tau=\theta$ implies that, for any $h_1,h_2,...,h_n,h'_1,h'_2,...,h'_n\in H$, we have:
$$\alpha_i(\theta(\theta(\alpha_1(h_1,h'_1),\alpha_2(h_1,h'_1),...,\alpha_n(h_1,h'_1),h'_1),$$
$$\theta(\alpha_1(h_2,h'_2),\alpha_2(h_2,h'_2),...,\alpha_n(h_2,h'_2),h'_2)$$
$$...,\theta(\alpha_1(h_n,h'_n),\alpha_2(h_n,h'_n),...,\alpha_n(h_n,h'_n),h'_n)),$$
$$\theta(0,0,...,0,a),$$
$$\theta(h'_1,h'_2,...,h'_{n},a))\in H.$$\vskip+1mm
\noindent Applying equality (1.2) $n$ times, we obtain that
$$\alpha_i(\theta(h_1,h_2,...,h_n,a),\theta(h'_1,h'_2,...,h'_{n},a))\in H,$$
and inclusion (2.9) is proved. 
\vskip+1mm

(iii)$\Rightarrow$(iv): To show inclusion (2.9) we consider (2.8) for $\tau=\alpha_j$ and $a_1=a_2=a$. Then we obtain 
$$\alpha_i(\alpha_j(\theta(H,H,...,H,a),\theta(H,H,...,H,a)),0)\subseteq H,$$
\noindent for any $i$ and $j$. Equality (1.2) implies that, for any $h_1,h_2,...,h_n,$ $h'_1,h'_2,...,h'_n\in H$, we have
$$\alpha_i(\theta(h_1,h_2,...,h_n,a),\theta(h'_1,h'_2,...,h'_n,a))=$$
$$\theta(\alpha_1(\alpha_i(\theta(h_1,h_2,...,h_n,a),\theta(h'_1,h'_2,...,h'_n,a)),0),$$
$$\alpha_2(\alpha_i(\theta(h_1,h_2,...,h_n,a),\theta(h'_1,h'_2,...,h'_n,a)),0),...,$$
$$\alpha_n(\alpha_i(\theta(h_1,h_2,...,h_n,a),\theta(h'_1,h'_2,...,h'_n,a)),0),0,).$$
\noindent Therefore, this element belongs to $H$.

The implication (iv)$\Rightarrow$(v) follows from equality (2.2), while (v)$\Rightarrow$(vi) follows from Lemma 2.4.

(vi)$\Rightarrow$(vii): Let us show that for any $\tau \in \cF$ and any $j$, we have 
$$\tau(a_1 , a_2 ,...,a_{j-1},\left[  a_{j}\right],a_{j+1}...a_k )\subseteq$$ 
\begin{equation}
\left[ \tau(a_1,a_2,...,a_{j-1},a_{j},a_{j+1},...,a_{k})\right].
\end{equation}
\noindent Inclusion (2.10) and equality (2.11) imply that, for any $$b\in \tau(a_1,a_2 ,...,a_{j-1},\left[  a_{j})\right],a_{j+1},...,a_k )=$$
$$\tau(a_1,a_2,...,a_{j-1},\theta(H,H,...,H,a_j),a_{j+1},...,a_{k}),$$
\noindent and any $i$ we have
\begin{equation} 
\alpha_i(b,\tau(a_1,a_2,...,a_{k}))\in H.
\end{equation}
\noindent Hence
$$b=\theta(\alpha_1(b,\tau(a_1,a_2,...,a_{k})),\alpha_2(b,\tau(a_1,a_2,...,a_{k})),$$
$$...,\alpha_n(b,\tau(a_1,a_2,...,a_{k})), \tau(a_1,a_2,...,a_{k}))\in$$ 
$$\theta(H,H,...,H,\tau(a_1,a_2,...,a_{k}))=$$
$$\left[ \tau(a_1,a_2,...,a_{k})\right],$$
\noindent and inclusion (2.12) is proved. Thus $\sim_H$ is a congruence. Lemma 2.3 implies that $H$ is the kernel of the congruence $\sim_H$.

The implication (vii)$\Rightarrow$(i) is well-known. 

\end{proof}
Observe that
inclusion (2.9) is satisfied for any subset $H$ of a $\mathbb{V}$-algebra $A$ which is right-cancellable in the sense of \cite{Z1}. Recall this means that for any $a_1,a_2,...,$
$a_n,a'_1,a'_2,...,a'_n,b,b'\in A$ and any $i$ with $1\leqslant i\leqslant n$ one has 
$$\alpha_i(\theta(a_1,a_2,...,a_n,b),\theta(a'_1,a'_2,...,a'_n,b))=$$
$$\alpha_i(\theta(a_1,a_2,...,a_n,b'),\theta(a'_1,a'_2,...,a'_n,b')).$$
\vskip+4mm

We say that a set $T$ of ideal terms determines ideals in a variety $\mathbb{V}$ if "any ideal term" in the definition of an ideal can be replaced by "any ideal term from $T$". Theorem 2.5 implies

\begin{theo}
Let $\mathbb{V}$ be a BIT speciale variety. Then each of the following sets of terms (being ideal terms in the variables $y$ with or without indices) determines ideals in $\mathbb{V}$: 

(i)  the set of the terms \begin{equation}
\theta(y_1,y_2,...,y_n, y_{n+1}),
\alpha_i(y_1,y_2),
\end{equation} 
%$$\alpha_i(y_i,y_2),$$
\noindent and
\begin{equation}
\alpha_i(\tau(\theta(y_{11},y_{12},...,y_{1n},x_1),\theta(y_{21},y_{22},...,y_{2n},x_2),
\end{equation}
$$...,\theta(y_{k1},y_{k2},...,y_{kn},x_{k})),\tau(x_1,x_2,...,x_k)),$$

\noindent for all operations $\tau$ from the set $\cF\cup \lbrace \theta \rbrace$, and all $i$ with $1\leqslant i\leqslant n$; \vskip+2mm

(ii) the set of the terms \begin{equation}
0,
\theta(y_1,y_2,...,y_n, y_{n+1}),
\alpha_i(y,0),
\end{equation} 
%$$\alpha_i(y_i,y_2),$$
\noindent and terms (2.15) for all operations $\tau$ from the set $\cF\cup \lbrace \theta \rbrace \cup \lbrace \alpha_j\vert 1\leqslant j\leqslant n\rbrace$ and all $i$ with $1\leqslant i\leqslant n$; \vskip+2mm

(iii) the set of the terms  \begin{equation}
\theta(y_1,y_2,...,y_n, y_{n+1}),
\alpha_i(y,0),
\end{equation} 
\begin{equation}
\alpha_i(\theta(y_1,y_2,...,y_n,x),\theta(y_{n+1},y_{n+2},...,y_{2n},x)),
\end{equation}
\noindent and terms (2.15) for all operations $\tau$ from the signature $\cF$ and all $i$ with $1\leqslant i\leqslant n$;\vskip+2mm

(iv) the set of the terms (2.18) and 
\begin{equation}
\theta(y_1,y_2,...,y_n, y_{n+1}),
\alpha_i(y,0),
\end{equation} 
\begin{equation}
\alpha_i(\tau(x_1, x_2,....,x_{j-1},\theta(y_1,y_2,...,y_n,x_j),x_{j+1},...
\end{equation}
$$...,x_{k}),\tau(x_1,x_2,...,x_{k}));$$
\noindent for all operations $\tau$ from the signature $\cF$ and all $i$, $j$ with $1\leqslant i\leqslant n, 1\leq j\leq k$;\vskip+2mm

If $\mathbb{V}$ is semi-abelian, then terms (2.14) in (i), (2.16) in (ii), (2.17) in (iii), and (2.19) in (iv) can be replaced by the terms
\begin{equation}
\tau(y_1,y_2,...,y_n),
\end{equation}
\noindent for all operations $\tau$ from the signature $\cF$ (including $0$).
\end{theo}

Observe that if $\cF$ is finite, then obviously all four term sets in Theorem 2.6 also are finite.
\vskip+2mm

Finally let us give several easy corollaries of Theorem 2.5. Consider two varieties $\mathbb{V}$ and $\mathbb{V'}$  with signatures resp. $\cF$ and $\cF'$, and  sets of identities resp. $\Sigma$ and $\Sigma'$. Assume that $\cF\subseteq \cF'$ and $\Sigma\subseteq \Sigma'$. If $\mathbb{V}$ is a BIT speciale variety, then obviously $\mathbb{V'}$ is also of this kind with the same terms $0$, $\theta$, and $\alpha_i$. Theorem 2.5 implies 
  
  \begin{cor}
Let $\mathbb{V}$ be a BIT speciale variety, $A'$ be a $\mathbb{V'}$-algebra, and $H$ be a subset of $A'$. The following conditions are equivalent:\vskip+1mm

(i) $H$ is an ideal of $A'$ in the variety $\mathbb{V'}$;\vskip+1mm

(ii) $H$ is an ideal of $A'$ in the variety $\mathbb{V}$ and inclusion (2.8) is satisfied for any $a_1, a_2,...,a_k\in A'$ and any $\tau\in \cF'\setminus \cF$.\vskip+1mm

(iii) $H$ is an ideal of $A'$ in the variety $\mathbb{V}$ and inclusions (2.10) are satisfied for any $a_1, a_2,...,a_k\in A'$, any $\tau\in \cF'\setminus \cF$, and any $i$, $j$ with $1\leqslant i\leqslant n, 1\leq j\leq k$.
\end{cor}

\begin{cor}
Let $\mathbb{V}$ be a BIT speciale variety, and $A'$ be a $\mathbb{V'}$-algebra. Let $\cF=\cF'$. A subset $H$ of $A'$ is an ideal in the variety $\mathbb{V'}$ if and only if it is an ideal of $A'$ in the variety $\mathbb{V}$.
\end{cor}

Observe that the condition that $\mathbb{V}$ is BIT speciale is essential in both Corollaries 2.7 and 2.8. Removing this condition we still have that any subset which is an ideal in $\mathbb{V}'$ is an ideal in $\mathbb{V}$ too, but the converse is not true even if $\cF=\cF'$ (for a counterexample take the variety $\mathbb{V}'$ of groups, and the variety $\mathbb{V}$ with the same signature, but the empty set of identities).\vskip+2mm 

Corollaries 2.7 implies
\begin{cor}
Let $\mathbb{V}$ be a BIT speciale variety. 

(a) The union of any set of ideal terms determining ideals of $\mathbb{V}$ with the set of terms (2.15) for $\tau\in \cF'\setminus \cF$ determines ideals in $\mathbb{V'}$.

(b) The union of any set of ideal terms determining ideals of $\mathbb{V}$ with the set of terms (2.20) for $\tau\in \cF'\setminus \cF$ and $1\leqslant i\leqslant n, 1\leq j\leq k$  determines ideals in $\mathbb{V'}$.
\end{cor}

\begin{cor}
 If $\cF= \cF'$, any set of ideal terms determining ideals in $\mathbb{V}$ determines ideals in $\mathbb{V'}$ too.
\end{cor}

\section{Examples}
For all varieties considered in this section, the signature/algebraic theory of which contains the constant $e$, the multiplication operation $\cdot$ and the division operation /, there are the following BIT speciale terms:
\begin{equation}
 0=e, \theta(x,y)=xy, \alpha(x,y)=x/y.
 \end{equation}
 
 Throughout the section we identify the equivalent terms (like, for instance, the terms $y/e$ and $y$ in the variety of groups).\vskip+2mm
  
For the variety of groups, the set (iv) from Theorem 2.6 consists of the terms: 
\begin{equation} 
 y_1y_2, xyx^{-1}, x^{-1}y^{-1}x,
\end{equation}
\noindent and the identity term $y$ which can be ignored (so below we will not mention it).
\vskip+2mm

For the variety of rings the set (iv) from Theorem 2.6 consists of the terms $y_1+y_2$, $-y$, $xy$ and $yx$. 
\vskip+2mm

More generally, consider the variety $\Omega$-$Grp$ of groups with operators from a set $\Omega$ (three criteria for a variety of universal algebras to be of this kind, formulated in BIT speciale terms are given in our recent papers \cite{Z} and \cite{Z1}). The set (iv) from Theorem 2.6 consists of terms (3.2) and 
 \begin{equation}
\omega(x_1,x_2,...,yx_i,...,x_k)(\omega(x_1,x_2,...,x_k))^{-1},
\end{equation}
\noindent for all $\omega\in \Omega$ and all $i$ with $1\leqslant i\leqslant n$. Note that a set of ideal terms determining ideals in the variety $\Omega$-$Grp$ was given by Higgings in \cite{Hi}. It consists of the well-known terms determining ideals of groups (i.e. normal subgroups) and the terms which can be obtained from (3.3) by permuting the factors, i.e. the terms
 \begin{equation}
(\omega(x_1,x_2,...,x_k))^{-1}\omega(x_1,x_2,...,yx_i,...,x_k).
\end{equation}

When $\omega$ is a unary distributive term
(like the multiplication by a scalar in the case of algebras over a ring), term (3.3) takes the form
$\omega(y)$. 
\vskip+2mm

Recall that a loop can be defined as a set $A$ equipped with three binary operations $\cdot$, $/$, and $\setminus$ and a constant $e$ satisfying the identities:
\begin{equation}
x/x=e,
\end{equation}
\begin{equation}
(x/y)y=x,
\end{equation}
\begin{equation}
(xy)/y=x.
\end{equation}
\begin{equation}
y(y\setminus x)=x,
\end{equation}
\begin{equation}
y\setminus (yx)=x,
\end{equation}
\begin{equation}
x\setminus x=e.
\end{equation}
\noindent Hence, for the variety of loops, the set (i) from Theorem 2.6 consists of the terms
\begin{equation}
 y_1y_2, y_1/y_2,
\end{equation}
\noindent 
\begin{equation}
((y_1x_1)(y_2x_2))/(x_1x_2),
\end{equation}
\begin{equation}
((y_1x_1)/(y_2x_2))/(x_1/ x_2).
\end{equation}
\begin{equation}
((y_1x_1)\setminus(y_2x_2))/(x_1\setminus x_2).
\end{equation}
\vskip+1mm
\noindent while the set (iv) of Theorem 2.6 consists of the following terms:
\begin{equation}
(y_1x_1)/(y_2x_2),
\end{equation}
\begin{equation} 
y_1y_2, 
\end{equation} 
\begin{equation}
((y_1x_1)x_2)/(x_1x_2),
\end{equation}
\begin{equation}
(x_1(y_1x_2))/(x_1x_2),
\end{equation}
\begin{equation}
((y_1x_1)/x_2)/(x_1/x_2),
\end{equation}
\begin{equation}
(x_1\setminus (yx_2))/(x_1\setminus x_2).
\end{equation}
\begin{equation}
(x_1/(yx_2))/(x_1/x_2),
\end{equation}
\begin{equation}
((y_1x_1)\setminus x_2)/(x_1\setminus x_2)
\end{equation}

\noindent and the term $y$ which can be ignored, so that, like in the case of groups, we will not count it. Recall that a set of ideal terms determining ideals of loops was found by Bruck in \cite{Br}. It consists of terms (3.16)-(3.20) and the terms 
\begin{equation}
y_1/y_2, y_1\setminus y_2.
\end{equation}\vskip+2mm

Let $\Omega$-$Loop$ be the variety of loops with operators from a set $\Omega$. The set (i) of Theorem 2.6 consists of terms (3.11)-(3.14) and 
\begin{equation}
\omega(y_1x_1,y_2x_2,...,y_kx_k)/\omega(x_1,x_2,...,x_k),
\end{equation}
\noindent for all $\omega\in \Omega$. The set (iv) from Theorem 2.6 consists of terms (3.15)-(3.22) and
 \begin{equation}
\omega(x_1,x_2,...,yx_i,...,x_k)/\omega(x_1,x_2,...,x_k),
\end{equation}
\noindent for all $\omega\in \Omega$ and all $i$ with $1\leq i\leq n$. Another set of terms determining ideals in the variety $\Omega$-$Loop$ was found by Higgins in \cite{Hi}. It consists of terms (3.16)-(3.20), (3.23) and (3.25).\vskip+2mm

Recall that a (right) semi-loop is a set $A$ equipped with two binary operations $\cdot$, $/$, and a constant $e$ which satisfies identities (3.5)-(3.7). $A$ is called a (right) divisible involutory groupoid  if it satisfies identities (3.5)-(3.6) \cite{H}. For both varieties, the set (i) of Theorem 2.6 consists of terms (3.11)-(3.13), which coincides with the set of terms found by Hala\v{s} \cite{H} for divisible involutory groupoids, and almost coincides with the set of terms found by B\v{e}lohl\'{a}vek and Chajda \cite{BC} for semi-loops. The latter one consists of term (3.11) and the terms which can be obtained from (3.13) and (3.14) by permuting the external numerators and denominators, i.e. the terms
\begin{equation}
(x_1x_2)/((y_1x_1)(y_2x_2)),
\end{equation}
\begin{equation}
(x_1/ x_2)/((y_1x_1)/(y_2x_2)).
\end{equation}
\vskip+4mm

Recall that the BIT speciale terms for the variety of Heyting semi-lattices and Heyting algebras were found by Johnstone \cite{J} (observe that $0$ is the constant $\mathsf{T}$, so that $0$-ideals are filters in this case). For these varieties the terms from any set given by Theorem 2.6 take very cumbersome forms, and we do not give them here. For the variety of Boolean algebras, we have at least two collections of BIT speciale terms -- one of them is the same as that of the variety of Heyting algebras. The second one is obtained from the group structure of Boolean algebras. For both collections of BIT speciale terms, some of terms from any set given by Theorem 2.6 are cumbersome.

\section{Acknowledgments}

Financial support from  Shota Rustaveli  Georgian National Science Foundation
(Ref.: FR-18-10849) is gratefully acknowledged.

\vskip
+5mm
\vskip+3mm

\textit{Author's address:
Dali Zangurashvili, A. Razmadze Mathematical Institute of Tbilisi State University,}
\textit{6 Tamarashvili Str., Tbilisi 0177, Georgia, e-mail: dali.zangurashvili@tsu.ge}

\end{document}